\documentclass[12pt,amscd]{amsart}
\usepackage[all]{xy}
\usepackage{graphicx}
\usepackage{amsmath,amsxtra,amssymb,latexsym, amscd,amsthm}
\usepackage{indentfirst}
\usepackage[mathscr]{eucal}
\usepackage[pagebackref=true]{hyperref}

\newtheorem{thm}{Theorem}[section]

\newtheorem{lem}[thm]{Lemma}

\theoremstyle{definition}
\newtheorem{defn}[thm]{Definition}

\newtheorem{rem}[thm]{Remark}

\newtheorem{conj}{Conjecture}

\numberwithin{equation}{section}

\DeclareMathOperator{\NN}{\mathbb {N}}

\DeclareMathOperator{\lk}{lk}

\DeclareMathOperator{\rstab}{rstab}

\DeclareMathOperator{\ord}{ord}

\DeclareMathOperator{\supp}{supp}

\DeclareMathOperator{\Inter}{Inter}
\DeclareMathOperator{\reg}{reg}

\def\D {\Delta}

\def\a {\mathbf a}
\def\b {\mathbf b}
\def\x {\mathbf x}

\def\m {\mathfrak m}
\def\n {\mathfrak n}
\def\F {\mathfrak F}

\def\k {\mathrm{k}}
\def\h {\widetilde{H}}

\begin{document}

\title[Powers Gap-free graphs] {Characterization of graphs whose a small power of their edge ideals has a linear free resolution}

\author{Nguyen Cong Minh}
\address{School of Applied Mathematics and Informatics, Hanoi University of Science and Technology, 1 Dai Co Viet, Hanoi, Vietnam}
\email{minh.nguyencong@hust.edu.vn}

\author{Thanh Vu}
\address{Hanoi University of Science and Technology, 1 Dai Co Viet, Hanoi, Vietnam}
\email{vuqthanh@gmail.com}


\date{}

\dedicatory{Dedicated to Professor L\^e Tu\^an Hoa on the occasion of his 65th birthday}
\commby{}
\maketitle
\begin{abstract}
    Let $I(G)$ be the edge ideal of a simple graph $G$. We prove that $I(G)^2$ has a linear free resolution if and only if $G$ is gap-free and $\reg I(G) \le 3$. Similarly, we show that $I(G)^3$ has a linear free resolution if and only if $G$ is gap-free and $\reg I(G) \le 4$. We deduce these characterizations from a general formula for the regularity of powers of edge ideals of gap-free graphs
    $$\reg(I(G)^s) = \max(\reg I(G) + s-1,2s),$$
    for $s =2,3$.
\end{abstract}

\maketitle

\section{Introduction}
\label{sect_intro}
Let $G$ be a simple graph on $n$ vertices. Denote $I(G)$ the edge ideal of $G$, i.e., squarefree monomial ideal in a polynomial ring $S=\k[x_1,...,x_n]$ over a field $\k$ generated by $x_ix_j$ where $\{i,j\}$ is an edge of $G$. Fr\"{o}berg \cite{F} proved that $I(G)$ has a linear resolution if and only if $G$ is the complement of a chordal graph. Herzog, Hibi, and Zheng \cite{HHZ} showed that in this case, all powers of $I(G)$ have a linear resolution as well. Since the establishment of the asymptotic linearity of the regularity of powers of ideals in \cite{CHT, Kod}, this was the first explicit computation of the regularity of powers of ideals for a large class of graphs. Given a graph $G$, we know that there exist natural numbers $q$ and $b(G)$ depending only on $G$ such that $\reg I(G)^s = 2s + b(G)$ for all $s \ge q$. The smallest such natural number $q$ is called the regularity stabilization index of $I(G)$, denoted by $\rstab(G)$. In \cite{MV3}, we construct examples of graphs $G$ such that $b(G)$ depends on the characteristic of the base field. Hence for general graphs, one cannot hope to have a combinatorial characterization of $b(G)$. On the other hand, Ha, Francisco, and Van Tuyl \cite[Proposition 1.3]{NP} showed that if a power of $I(G)$ has a linear relation, then $G$ must be gap-free, i.e., the induced matching number of $G$ is one and asked if $I(G)^s$ have a linear resolution for $s \ge 2$. In \cite{NP}, Nevo and Peeva constructed a gap-free graph $G$ with $\reg I(G) = 4$ and $I(G)^2$ does not have a linear resolution. Computation showed that for this particular example, $I(G)^3$ and $I(G)^4$ have a linear resolution, and they asked whether $b(G) = 0$ if and only if $G$ is gap-free. With current techniques, this problem is still challenging. As stated in \cite{EHHM}, one of the recent targets is to classify graphs $G$ for which $I(G)^2$ has a linear resolution. We complete this in:

\begin{thm}\label{characterization_second_power} Let $I(G)$ be the edge ideal of a simple graph $G$. Then $I(G)^2$ has a linear resolution if and only if $G$ is gap-free and $\reg I(G) \le 3$.
\end{thm}
 
We also prove that a similar characterization holds for the third power
\begin{thm}\label{characterization_third_power} Let $I(G)$ be the edge ideal of a simple graph $G$. Then $I(G)^3$ has a linear resolution if and only if $G$ is gap-free and $\reg I(G) \le 4$.
\end{thm}

By the above-mentioned result of Ha, Francisco, and Van Tuyl, it suffices to restrict to gap-free graphs in order to classify graphs for which a power of their edge ideals has a linear resolution. Indeed, we deduce our characterizations from the following result on the regularity of the second and third powers of edge ideals of gap-free graphs.

\begin{thm}\label{regularity_power_gap_free_small_power} Let $I(G)$ be the edge ideal of a gap-free graph $G$. Then 
$$\reg I(G)^s = \max (\reg I(G) + s - 1, 2s),$$
for $s = 2,3$.
\end{thm}

Other sources of motivation for studying the regularity of edge ideals of gap-free graphs and their powers are the following. Dao, Huneke, and Schweig \cite{DHS} proved that the regularity of $I(G)$ is bounded above by a log function of the number of variables. On the other hand, based on the result of Januszkiewicz and Swiatkowski \cite{JS}, Constantinescu, Kahle, and Varbaro \cite{CKV} constructed gap-free graphs whose regularity is an arbitrary positive integer. Theorem \ref{regularity_power_gap_free_small_power} implies that for gap-free graphs of regularity at least $4$ have $\rstab(G) \ge 3$. We note that, in any current known results about the regularity of powers, we have $\rstab(G) \le 2$. Furthermore, by the construction of Lutz and Nevo \cite{LN}, there are infinitely many gap-free graphs whose regularity is a specified positive number.

Partly motivated by the question of Nevo and Peeva, Barnejee \cite{B} developed a technique to bound the regularity of powers by bounding the regularity of certain colon ideals. Indeed, his technique has played an important role in much subsequent work on the regularity of powers of edge ideals, e.g. \cite{AB, ABS, BHT, BN, Er, JNS}. This technique was further modified to apply to the study of the regularity of symbolic powers of edge ideals, e.g. \cite{F1, F2, F3, JK}.

For a radical ideal $I$ in $S$, $I^{(s)}$ denotes the $s$-th symbolic powers of $I$. In \cite{MNPTV, MV2}, we develop another approach toward computing the regularity of powers of edge ideals of graphs and comparing it with the symbolic powers via the study of the degree complexes (see subsection \ref{subsec_degree_complexes} for more details). Using these techniques, we proved that $\reg I(G)^s = \reg I(G)^{(s)}$ for $s \le 3$ and compute $\reg I_\Delta^s$ for all one-dimensional simplicial complex $\Delta$. We further demonstrate its usefulness in this paper in proving Theorem \ref{regularity_power_gap_free_small_power}. First, we prove the following lower bound for arbitrary squarefree monomial ideals.
\begin{lem}\label{lowerbound}
Let $I$ be a squarefree monomial ideal. Then 
$$\reg I + s - 1 \le \min(\reg I^s,\reg I^{(s)}).$$
\end{lem}
Note that this result is of independent interest for squarefree monomial ideals, as there are monomial ideals $I$ with $\reg I > \reg I^2$ (see \cite[Remark 5.9]{NV}). 

To establish Theorem \ref{regularity_power_gap_free_small_power}, by \cite[Theorem 1.1]{MNPTV} and Lemma \ref{lowerbound}, it remains to prove that $\reg I(G)^{(s)} \le \max(\reg I(G) + s-1, 2s)$ for gap-free graphs $G$ and $s = 2,3$. To accomplish that, we use a recent result of Hien and Trung \cite{HiTr} to bound the degree of extremal exponents of symbolic powers of $I(G)$ (see subsection \ref{subsec_degree_complexes} for the definition of extremal exponents). For each extremal exponent $\a$ of $I(G)^{(s)}$ for $s = 2,3$, we expressed $\sqrt{I^{(s)}:x^\a}$ as the intersection of sums of ideals of the form $I:x_j$. Hence, the corresponding degree complexes is the union of simplicial cones. Chasing along the long exact sequence of homology groups, we deduce the desired conclusion.

Based on the evidences in Theorem \ref{regularity_power_gap_free_small_power} and Lemma \ref{lowerbound}, we propose the following:

\begin{conj}\label{conj_power_gap_free} Let $I$ be the edge ideal of a gap-free graph $G$. Then 
$$\reg (I^s) =\reg (I^{(s)}) = \max(\reg I + s-1, 2s),$$
for all $s \ge 2$. 
\end{conj}

Conjecture \ref{conj_power_gap_free} implies that the stabilization index of edge ideals of graphs could be an arbitrary natural number.

Now we explain the organization of the paper. In Section \ref{sec_basic}, we recall some notation and basic facts about graphs and their edge ideals, the symbolic powers of squarefree monomial ideals, the degree complexes, and Castelnuovo-Mumford regularity. In Section \ref{sec_second_power}, we prove Theorem \ref{regularity_power_gap_free_small_power} for $s=2$ and deduce Theorem \ref{characterization_second_power}. In Section \ref{sec_third_power}, we prove Theorem \ref{regularity_power_gap_free_small_power} for $s=3$ and deduce Theorem \ref{characterization_third_power}.

\section{Castelnuovo-Mumford regularity, symbolic powers, and degree complexes}\label{sec_basic}
In this section, we recall some definitions and properties concerning simplicial complexes and Stanley-Reisner correspondence, graphs and their edge ideals, Castelnuovo-Mumford regularity, the symbolic powers of a squarefree monomial ideal, and the degree complexes of a monomial ideal. We then prove Lemma \ref{lowerbound}. The interested readers are referred to (\cite{BH, D, E, S}) for more details.

\subsection{Simplicial complexes and Stanley-Reisner correspondence} 

Let $\Delta$ be a simplicial complex on $[n]=\{1,\ldots, n\}$ that is a collection of subsets of $[n]$ closed under taking subsets. We put $\dim F = |F|-1$, where $|F|$ is the cardinality of $F$. The dimension of $\Delta$ is $\dim \Delta = \max \{ \dim F \mid F \in \Delta \}$.  The set of its maximal elements under inclusion, called by facets, is denoted by $\F(\Delta)$.

A simplicial complex $\D$ is called a cone over $x\in [n]$ if $x\in B$ for any $B\in \F(\Delta)$.

For a face $F\in\Delta$, the link of $F$ in $\Delta$ is the subsimplicial complex of $\Delta$ defined by
$$\lk_{\Delta}F=\{G\in\Delta \mid  F\cup G\in\Delta, F\cap G=\emptyset\}.$$

For each subset $F$ of $[n]$, let $x_F=\prod_{i\in F}x_i$ be a squarefree monomial in $S$. We now recall the Stanley-Reisner correspondence.

\begin{defn}For a squarefree monomial ideal $I$, the Stanley-Reisner complex of $I$ is defined by
$$ \Delta(I) = \{ F \subseteq [n] \mid x_F \notin I\}.$$

For a simplicial complex $\Delta$, the Stanley-Reisner ideal of $\Delta$ is defined by
$$I_\Delta = (x_F \mid  F \notin \Delta).$$
The Stanley-Reisner ring of $\Delta$ is $\k[\Delta] =  S/I_\Delta.$
\end{defn}
From the definition, it is easy to see the following:
\begin{lem}\label{cone} Let $I, J$ be squarefree monomial ideals of  $S = \k[x_1,\ldots, x_n]$. Then 
\begin{enumerate}
    \item $I \subseteq J$ if and only if $\Delta(I) \supseteq \Delta(J)$.
    \item $\Delta(I)$ is a cone over $t \in [n]$ if and only if $x_t$ does not divide any minimal generator of $I$.
    \item $\Delta(I + J) = \Delta(I) \cap \Delta(J).$
    \item $\Delta(I \cap J) = \Delta(I) \cup \Delta(J).$
\end{enumerate}
\end{lem}

Each simplicial complex $\Delta$ gives rise to an augmented oriented chain complex $(C(\Delta),\epsilon)$ over $\k$. 

\begin{defn} The $q$-th reduced homology group of $\Delta$ with coefficients $\k$, denoted $\h_q(\Delta; \k)$ is defined to be the $q$-th homology group of the augmented oriented chain complex of $\Delta$ over $\k$.
\end{defn}
A simplicial complex $\D$ is called {\it acyclic} if $\h_i(\Delta;\k) = 0$ for all $i$.
\begin{rem} \begin{enumerate}
    \item If $\Delta$ is the empty complex (i.e., $\Delta=\{\emptyset\}$), then $\h_i(\Delta;\k) \neq 0$ if and only if $i = -1$.
    \item If $\Delta$ is a cone over some $t \in [n]$ or $\Delta$ is the void complex  (i.e., $\Delta=\emptyset$), then it is acyclic.
\end{enumerate}
\end{rem}
The following simple lemma \cite[Lemma 2.5]{MNPTV} will be useful later on.
\begin{lem}\label{lem_MayerVietoris} Let $\Delta$ be a simplicial complex on $[n]$ with $\h_{i-1}(\Delta;\k) \neq 0$ for some $i \ge 0$. Assume that $\Delta = \Gamma_1 \cup \Gamma_2$ is a decomposition of $\Delta$ as the union of two subsimplicial complexes. Then at least one of the homology groups  $\h_{i-1}(\Gamma_1;\k)$, $\h_{i-1}(\Gamma_2;\k)$, $\h_{i-2}(\Gamma_1 \cap \Gamma_2;\k)$ is non-zero.
\end{lem}

\subsection{Castelnuovo-Mumford regularity} 
Let $\m = (x_1,\ldots, x_n)$ be the maximal homogeneous ideal of $S = \k[x_1,\ldots, x_n]$, a standard graded polynomial ring over a field $\k$. For a finitely generated graded $S$-module $L$, the Castelnuovo-Mumford regularity (or regularity for short) of $L$ is defined to be
$$\reg(L) = \max\{i + j \mid H_{\m}^i(L)_j \ne 0\},$$
where $H^{i}_{\m}(L)$ denotes the $i$-th local cohomology module of $L$ with respect to $\m$. 

For a non-zero proper homogeneous ideal $J$ of $S$, we have $\reg(J)=\reg(S/J)+1$.

\subsection{Symbolic powers of squarefree monomial ideals} 
Let $I$ be a non-zero and proper homogeneous ideal of $S$. Let $\{P_1,\ldots,P_r\}$ be the set of the minimal prime ideals of $I$. Given a positive integer $s$, the $s$-th symbolic power of $I$ is defined by
$$I^{(s)}=\bigcap_{i=1}^r I^sS_{P_i}\cap S.$$

For an exponent $\a \in \NN^n$, $x^\a$ denotes the monomial $x_1^{a_1} \cdots x_n^{a_n}$. For a monomial $f$ in $S$, we denote $\frac{\partial^* (f)}{\partial^*(x^\a)}$ the $*$-partial derivative of $f$ with respect to $x^\a$, which is derivative without coefficients. We define 
$$I^{[s]} =  \left ( f \in S ~|~ \frac{\partial^* f }{\partial^* x^\a} \in I, \text{ for all } x^\a \text{ with } |\a| \le s -1 \right ),$$ the $s$-th $*$-differential power of $I$. When $I$ is a squarefree monomial ideal, the symbolic powers of $I$ are equal to the $*$-differential powers of $I$ \cite[Lemma 2.6]{MNPTV}. 

\begin{lem}\label{differential_criterion} Let $I$ be a squarefree monomial ideal. Then $I^{(s)} = I^{[s]}.$
\end{lem}

\subsection{Degree complexes}\label{subsec_degree_complexes}
Together with Nam, Phong, and Thuy, we show that the regularity of a monomial ideal can be computed in terms of its degree complexes as follows.

\begin{lem}\label{Key0}
Let $I$ be a monomial ideal in $S$. Then
\begin{multline*}
\reg(S/I)=\max\{|\a|+i \mid \a\in\NN^n,i\ge 0,\h_{i-1}(\lk_{\D_\a(I)}F;\k)\ne 0\\ \text{ for some $F\in \D_\a(I)$ with $F\cap \supp \a=\emptyset$}\},
\end{multline*}
where $\D_\a(I) = \Delta(\sqrt{I:x^\a})$ is the degree complex of $I$ in degree $\a$. In particular, if $I=I_\D$ is the Stanley-Reisner ideal of a simplicial complex $\D$ then
$\reg(\k[\D])=\max\{i \mid i\ge 0,\h_{i-1}(\lk_{\D}F;\k)\ne 0\text{ for some }F\in \D\}$.
\end{lem}
\begin{proof}
Follows from \cite[Lemma 2.12]{MNPTV} and \cite[Lemma 2.19]{MNPTV}.
\end{proof}

\begin{defn}\label{exdef} Let $I$ be a monomial ideal in $S$. 
\begin{enumerate}
    \item A pair $(\a,i) \in \NN^n\times\NN$ is called {\it an extremal pair} of $I$ if $\reg(S/I) = |\a| + i$ as in Lemma \ref{Key0}. 
    \item The exponent $\a$ in an extremal pair $(\a,i)$ is called an {\it extremal exponent} of $I$.
\end{enumerate} 
\end{defn}

For a monomial $f$ in $S$ and $i \in [n]$, $\deg_i(f) = \max( t\mid x_i^t \text{ divides } f)$ denote the degree of $x_i$ in $f$. For a monomial ideal $I$, $\rho_i(I)$ is defined by
$$\rho_i(I) = \max(\deg_i(u) \mid u \text{ is a minimal monomial generator of } I).$$

\begin{rem}\label{rem_extremal_set} By \cite[Remark 2.13]{MNPTV} and \cite[Lemma 2.6]{MV2}, we have
\begin{enumerate}
    \item Let $\a$ be an extremal exponent of $I$. Then $x^\a \notin I$ and $\Delta_\a(I)$ is not a cone over any $t\in \supp \a$. Furthermore, $\a$ belongs to the finite set
	$$\Gamma(I)=\{\a\in\NN^n~|~ a_j<\rho_j(I)\text{ for all } j=1,\ldots,n\}.$$
\item In particular, for a squarefree monomial ideal $I$, if $\a$ is an extremal exponent of $I^{(s)}$ then $a_i < s$. 
\end{enumerate}
\end{rem}

While our motivation is to study the regularity of the ordinary power $I^s$, studying the regularity of intermediate ideals lying between $I^s$ and $I^{(s)}$ gives us insight into understanding the regularity of $I^s$. We first recall the following
\begin{defn}
Let $J \subseteq K$ be monomial ideals in $S$. We denote $\Inter(J, K)$ the
set of intermediate ideals between $J$ and $K$ containing all monomial ideals $L$ such that $L = J + (f_1,\ldots, f_t)$ where $f_j$ are among minimal
generators of $K$.
\end{defn}

\begin{lem}\label{lem_low_degree_extremal} Let $I$ be a nonzero squarefree monomial ideal and $J \in \Inter(I^s, I^{(s)})$ be an intermediate ideal. Assume that $|\a| \le s-1$. Then $\Delta_\a(J) = \Delta(I)$.
\end{lem}
\begin{proof} By Lemma \ref{Key0}, it suffices to prove that $\sqrt{J:x^\a} = I$. Since $I \subseteq \sqrt{I^s:x^\a} \subseteq \sqrt{J:x^\a} \subseteq \sqrt{I^{(s)}:x^\a}$, it suffices to prove the equality for $J = I^{(s)}$, which follows directly from Lemma \ref{differential_criterion}. We note that Hoa and Trung have proved this equality for $J = I^{(s)}$ in \cite[Lemma 1.3]{HTr}. 
\end{proof}

Consequently, we deduce the following lower bound on the regularity of powers of ideals.

\begin{lem}\label{lowerbound_intermediate} Let $I$ be a nonzero squarefree monomial ideal. Then for any intermediate ideal $J \in \Inter(I^s,I^{(s)})$, we have 
$$\reg J \ge \reg I + s-1.$$
\end{lem}
\begin{proof} By using Hochster's formula (also see Lemma \ref{Key0}), we can choose $F\in\Delta(I)$ and $i\ge 0$ such that $\reg(R/I)=i$ and $\h_{i-1}(\lk_{\D(I)}F;\k)\ne 0$. Since $I$ is nonzero, $\supp F$ is a proper subset of $[n]$. Hence there exists $j \in [n]$ such that $j \notin F$. Let $\x^\a=x_j^{s-1}$. By Lemma \ref{lem_low_degree_extremal}, $\Delta_{\a }(J)=\Delta(I)$. Using  Lemma \ref{Key0}, we deduce that $\reg J \ge |\a| + \reg I = \reg I + s-1$ as required. 
\end{proof}

For symbolic powers of a squarefree monomial ideal, we have an upper bound for an extremal exponent as follows. 
\begin{thm}\label{extremal_degree_bound} Let $I$ be a nonzero squarefree monomial ideal and $s\ge 1$. Let $\a$ be an extremal exponent of $I^{(s)}$. Then $$|\a| \le \delta(I)(s-1),$$ where the constant $\delta(I)$ is defined as in \cite{DHNT}. In particular, $|\a| \le 2s-2$ for any extremal exponent of $I^{(s)}$ if $I$ is an edge ideal of a simple graph.
\end{thm}
\begin{proof} By assumption, there exists $i$ such that $\h_{i-1}(\lk_{\D_\a(I^{(s)})}F;\k)\ne 0$ for some $F\in \D_\a(I^{(s)})$ with $F\cap \supp \a=\emptyset$. Assume that $F=\{m+1,..., n\}$. Let $\b=(b_1,\ldots,b_n)$ be a vector with $b_i=a_i$ if $i=1,\ldots,m$ and $b_i=-1$, otherwise. Then, 
	$$\Delta_\b(I^{(s)})=\lk_{\Delta_\a(I^{(s)})}F \text{ and } G_b=F.$$
With the same proof of \cite[Theorem 2.2]{HiTr}, let $J = IS[x_i^{-1}~|~i\in F]\cap k[x_1,\ldots, x_m]$ and $\a'=(a_1,\ldots, a_m)$, then $\Delta_\b(I^{(s)})=\Delta_{\a'}(J^{(s)})$. It is noted that $|\a'|=|\a|$. From this,
$$H^{i}_{\n}(k[x_1,\ldots, x_m]/J^{(s)})_{\a'}\ne 0,$$ 
where $\n=(x_1,\ldots,x_m)$. By \cite[Theorem 2.2]{HiTr}, we have $$|\a|=|\a'|\le\delta(J)(s-1)\le\delta(I)(s-1).$$
Our the last statement is followed by \cite[Example 4.4]{DHNT}.
\end{proof}


\subsection{Graphs and their edge ideals}

Let $G$ denote a finite simple graph over the vertex set $V(G)=[n] = \{1,2,\ldots,n\}$ and the edge set $E(G)$. For a vertex $x\in V(G)$, let the neighbours of $x$ be the subset $N_G(x)=\{y\in V(G)~|~ \{x,y\}\in E(G)\}$, and set $N_G[x]=N_G(x)\cup\{x\}$. For a subset $U$ of the vertices set $V(G)$, $N_G(U)$ and $N_G[U]$ are defined by $N_G(U)=\cup_{u\in U}N_G(u)$ and $N_G[U]=\cup_{u\in U}N_G[u]$. If $G$ is fixed, we shall use $N(U)$ or $N[U]$ for short.

An independent set in $G$ is a set of pairwise non-adjacent vertices. 

A subgraph $H$ is called an induced subgraph of $G$ if for any vertices $u,v\in V(H)\subseteq V(G)$ then $\{u,v\}\in E(H)$ if and only if $\{u,v\}\in E(G)$.

A gap is a pair of disjoint edges forming an induced subgraph of $G$. A gap-free graph is one with no gaps. From the definition, we get the following lemma.

\begin{lem}\label{lem_coveringedge} Assume that $G$ is gap-free and $x_jx_k$ is an edge of $G$. Then the restriction of $G$ to $[n] \setminus N[\{j,k\}]$ has no edges.
\end{lem}

The edge ideal of $G$ is defined to be
$$I(G)=(x_ix_j~|~\{i,j\}\in E(G))\subseteq S.$$
For simplicity, we often write $i \in G$ (resp. $ij \in G$) instead of $i \in V(G)$ (resp. $\{i,j\} \in E(G)$). By abuse of notation, we also call $x_ix_j \in I(G)$ an edge of $G$.

\subsection{Radicals of colon ideals} We start with a simple observation.

\begin{lem}\label{radical_colon} Let $I$ be a monomial ideal in $S$ generated by the monomials $f_1, \ldots, f_r$ and $\a \in \NN^n$. Then $\sqrt{I:x^\a}$ is generated by $\sqrt{f_1/\gcd(f_1, x^\a)}, \ldots, \sqrt{f_r/\gcd(f_r,x^\a)}$.
\end{lem}
\begin{proof}
See \cite[Lemma 2.24]{MNPTV}.
\end{proof}
Assume now that $I = I(G)$ is the edge ideal of a graph. Let $f$ be a monomial. The $I$-order of $f$ is defined by 
$$\ord_I(f) = \max (t \mid f \in I^t).$$
From the definition, it is clear that if $g|f$, then $\ord_I(g) \le \ord_I(f)$. Let $\a \in \NN^n$ be an exponent, we now recall the following description of $\sqrt{I^s:x^\a}$ \cite[Lemma 2.18]{MV2}.

\begin{lem}\label{criterion_in_power}  Let $F$ be an independent set of $G$ and $\a \in \NN^n$ an exponent. Assume that
\begin{equation}
  \sum_{j\in N(F)} a_j + \ord_I \left ( \prod_{u \notin N[F]} x_u^{a_u} \right ) \ge s, \label{eq_in_power} 
\end{equation}
then $x_F \in \sqrt{I^s:x^\a}$. Conversely, if $x_F$ is a minimal generator of $\sqrt{I^s:x^\a}$ then \eqref{eq_in_power} holds.
\end{lem}

\section{Linear resolutions of the second power of edge ideals of graphs}\label{sec_second_power}
In this section, we prove Theorem \ref{regularity_power_gap_free_small_power} for $s=2$ and then deduce a characterization of graphs whose second powers have a linear resolution.

\begin{thm}\label{thmpow2} 
Let $I$ be the edge ideal of a gap-free graph $G$. Let $J \in \Inter(I^2,I^{(2)})$ be an intermediate ideal. Then 
$$\reg J = \reg I^2 = \reg I^{(2)} = \max (\reg I + 1, 4).$$
\end{thm}
\begin{proof}
By Lemma \ref{lowerbound_intermediate}, \cite[Theorem 1.4]{MV1}, and degree reason, it suffices to prove that $\reg I^{(2)} \le \max(\reg I + 1, 4).$ Let $(\a,i)$ be an extremal pair of $I^{(2)}$. Fix a face $F$ of $\Delta_\a(I^{(2)})$ such that $F \cap \supp \a = \emptyset$ and $\h_{i-1} (\lk_{\Delta_\a(I^{(2)})}F;\k) \neq 0$. By Lemma \ref{lem_low_degree_extremal}, Theorem \ref{extremal_degree_bound}, and Remark \ref{rem_extremal_set}, we may assume that $x^\a = x_1x_2$. First, we prove 
\begin{equation}\label{eq3_1}
    \sqrt{I^{(2)}:x^\a} = (I:x_1) \cap (I:x_2).
\end{equation}
\begin{proof}[Proof of Eq. \eqref{eq3_1}]
Let $f$ be a monomial in $I^{(2)} :x^\a$. By Lemma \ref{differential_criterion}, $x_1f$ and $x_2f$ belongs to $I$. Thus $f$ belongs to the right hand side. Since $(I:x_1) \cap (I:x_2)$ is radical, taking the radical, we deduce that the left hand side is contained in the right hand side. Conversely, assume that $f$ is a minimal monomial generator of $(I:x_1) \cap (I:x_2)$. Since $I \subseteq \sqrt{I^{(2)}:x^\a}$, we may assume that $\supp f$ is an independent set. By Lemma \ref{criterion_in_power} and the fact that $N(\supp f) \supseteq \{1,2\}$, $f \in \sqrt{I^2:x^\a} \subseteq \sqrt{I^{(2)}:x^\a}$.
\end{proof}
Thus, $\Delta_\a(I^{(2)}) = \Gamma_1 \cup \Gamma_2$ where $\Gamma_1$ and $\Gamma_2$ are the simplicial complexes corresponding to $I:x_1$ and $I:x_2$. Since $\supp F \cap \{1,2\} = \emptyset$, $\lk_{\Gamma_1}F$ and $\lk_{\Gamma_2}F$ are cones over $1$ and $2$ respectively. By Lemma \ref{lem_MayerVietoris}, we must have $\h_{i-2} (\lk_{\Gamma_1 \cap \Gamma_2}F;\k) \neq 0$. Furthermore, $\Gamma_1 \cap \Gamma_2$ is the simplicial complex corresponding to $I + N(\{1,2\})$. If $x_1x_2 \notin I$, then $\Gamma_1 \cap \Gamma_2$ is a cone over $\{1, 2\}$, so is $\lk_{\Gamma_1 \cap \Gamma_2} F$, which is a contradiction. Thus, we may assume that $x_1x_2 \in I$. In this case, since $G$ is gap-free, by Lemma \ref{lem_coveringedge}, $I + N(\{1,2\})$ is generated by variables. Hence, $\Gamma_1 \cap \Gamma_2 = \Delta(I + N[\{1,2\}])$ is a simplex. Thus, $i = 1$ and $\reg I^{(2)} \le |\a| + i + 1 =  4$.
\end{proof}

\begin{thm} Let $I$ be the edge ideal of a simple graph $G$. Then $I^2$ has a linear resolution if and only if $G$ is gap-free and $\reg I \le 3$.
\end{thm}
\begin{proof} Assume that $G$ is gap-free and $\reg I \le 3$. By Theorem \ref{thmpow2}, $\reg I^2 = 4$. Thus, $I^2$ has a linear free resolution.

Conversely, assume that $I^2$ has a linear free resolution. By \cite[Proposition 1.3]{NP}, $G$ is gap-free. Now by Theorem \ref{thmpow2}, we have $\reg I^2 = \max \{ \reg I + 1, 4\} = 4$. Thus, $\reg I \le 3$. The conclusion follows.
\end{proof}

\begin{rem} It is interesting to find  a combinatorial characterization of gap-free graphs $G$ with $\reg I(G) \le 3$.
\end{rem}

\begin{rem} By Lemma \ref{lowerbound_intermediate} and \cite[Theorem 1.2]{BN}, for any graph $G$, $\reg I(G)^2 \in \{\reg I(G) + 1, \reg I(G) + 2\}.$
\end{rem}

\section{Linear resolutions of the third power of edge ideals of graphs}\label{sec_third_power}
In this section, we prove Theorem \ref{regularity_power_gap_free_small_power} for $s=3$ and then deduce a characterization of graphs whose third powers have a linear resolution.

\begin{thm}\label{thmpow3} 
Let $I$ be the edge ideal of a gap-free graph $G$. Let $J \in \Inter(I^3,I^{(3)})$ be an intermediate ideal. Then 
$$\reg J = \reg I^3 = \reg I^{(3)} = \max (\reg I + 2, 6).$$
\end{thm}

By Lemma \ref{lowerbound_intermediate}, \cite[Theorem 1.4]{MV1}, and degree reason, it suffices to prove that $\reg I^{(3)} \le \max(\reg I + 2, 6).$ Let $(\a,i)$ be an extremal pair of $I^{(3)}$. Fix a face $F$ of $\Delta_\a(I^{(3)})$ such that $F \cap \supp \a = \emptyset$ and $\h_{i-1} (\lk_{\D_\a(I^{(3)})} F;\k) \neq 0$. By Lemma \ref{lem_low_degree_extremal} and Theorem \ref{extremal_degree_bound}, we may assume that $3 \le |\a| \le 4$. We separate the case $|\a| =3 $ and $|\a| = 4$ in the following lemmas.

\begin{lem}
Assume that $|\a| = 3$. Then $|\a| +i \le 5$. 
\end{lem}
\begin{proof}
By Remark \ref{rem_extremal_set}, there are two cases as follows.

\noindent {\textbf {Case 1.}} $x^\a = x_1^2 x_2$. First, we claim that 
\begin{equation}\label{eq4_1}
    \sqrt{I^{(3)}:x^\a} = (I:x_1) \cap (I:x_2).
\end{equation}
\begin{proof}[Proof of Eq. \eqref{eq4_1}] Let $f$ be a monomial in $I^{(3)}:x^\a$. By Lemma \ref{differential_criterion}, \begin{equation}
    \partial^*(x_1^2x_2f)/\partial^*(x_1^2) = x_2f \in I \text{ and } \partial^*(x_1^2x_2f)/\partial^* (x_1x_2) = x_1 f\in I.
\end{equation} 
Thus, $f \in (I:x_1) \cap (I:x_2)$. Taking radical, we deduce that the right hand side contains the left hand side. Now assume that $f$ is a minimal monomial generator of $(I:x_1) \cap (I:x_2)$ . Since $I \subseteq \sqrt{I^{(3)}:x^\a}$, we may assume that $\supp f$ is an independent set. By Lemma \ref{criterion_in_power} and the fact that $N(\supp f) \supseteq \{1,2\}$, $f \in \sqrt{I^3:x^\a} \subseteq \sqrt{I^{(3)}:x^\a}$.
\end{proof}

Thus, $\Delta_\a(I^{(3)}) = \Gamma_1 \cup \Gamma_2$, where $\Gamma_1$, $\Gamma_2$ are simplicial complexes corresponding to $I:x_1$ and $I:x_2$ respectively. With arguments similar to the proof of Theorem \ref{thmpow2}, we deduce that $i \le 1$ and $|\a| +i \le 4$.

\noindent{\textbf{Case 2.}} $x^\a = x_1 x_2 x_3$. We claim that 
\begin{equation}
    \sqrt{I^{(3)}:x^\a} = (I:x_1) \cap (I:x_2) \cap (I:x_3).
\end{equation}
The proof is similar to that of Eq. \eqref{eq4_1}. Thus, $\Delta_\a(I^{(3)}) = \Gamma_1 \cup \Gamma_2 \cup \Gamma_3$ where $\Gamma_1, \Gamma_2, \Gamma_3$ are simplicial complexes corresponding to $I:x_1, I:x_2, I:x_3$ respectively. Since $F \cap \supp \a = \emptyset$, $\lk_{\Gamma_1}F$, $\lk_{\Gamma_2}F$, and $\lk_{\Gamma_3}F$ are cones over $1,2$, and $3$ respectively. By Lemma \ref{lem_MayerVietoris}, either $\tilde H_{i-1} (\lk_{\Gamma_1 \cup \Gamma_2} F; \k) \neq 0$ or $\tilde H_{i-2} (\lk_{(\Gamma_1 \cup \Gamma_2) \cap \Gamma_3}F; \k) \neq 0$. For the first case, with arguments similar to the proof of Theorem \ref{thmpow2}, we deduce that $i \le 1$. Hence $|\a| + i \le 4$. Thus, we may assume that $\h_{i-2}(\lk_{(\Gamma_1 \cup \Gamma_2) \cap \Gamma_3}F;\k) \neq 0$. By Lemma \ref{lem_MayerVietoris} and symmetry, we have two subcases.

Subcase 2.a. $\h_{i-2} (\lk_{\Gamma_1 \cap \Gamma_3} F;\k) \neq 0$. With an argument similar to the proof of Theorem \ref{thmpow2}, we deduce that $13$ is an edge of $G$ and $\Gamma_1 \cap \Gamma_3$ is a simplex. Thus, $i = 1$ and $|\a| +i = 4$.

Subcase 2.b. $\h_{i-3} (\lk_{\Gamma_1 \cap \Gamma_2 \cap \Gamma_3} F;\k) \neq 0$. Note that $\Gamma_1 \cap \Gamma_2 \cap \Gamma_3 = \Delta(I + N(\{1,2,3\}))$. Since $F \cap \{1,2,3\} = \emptyset$, for $\Gamma_1 \cap \Gamma_2 \cap \Gamma_3$ not be a cone over $j$ for $j \in \{1,2,3\}$ we must have $j \in N(\{1,2,3\})$. In this case $I + N(\{1,2,3\}) = I + N[\{1,2,3\}]$ is generated by variables by Lemma \ref{lem_coveringedge}. Hence, $\Gamma_1 \cap \Gamma_2 \cap \Gamma_3$ is a simplex. Thus, $i -3 = -1$. Therefore, $|\a| +i = 5$. The conclusion follows.
\end{proof}

\begin{lem} Assume that $|\a| = 4$. Then $|\a| + i \le 5$.
\end{lem}
\begin{proof} By Remark \ref{rem_extremal_set}, there are three cases as follows.

\noindent \textbf{Case 1.} $x^\a = x_1^2 x_2^2$. First, we have $\sqrt{I^{(3)}:x^\a} = (I:x_1) \cap (I:x_2)$. With an argument similar to the proof of Theorem \ref{thmpow2}, we deduce that $i = 1$. Hence $|\a| +i = 5$.

\noindent \textbf{Case 2.} $x^\a = x_1^2 x_2 x_3$. First, we claim that
\begin{equation}\label{eq4_3}
    \sqrt{I^{(3)} : x^\a} = (I : x_1) \cap (I:(x_2x_3))
\end{equation}
\begin{proof}[Proof of Eq. \eqref{eq4_3}] Let $f$ be a monomial in $I^{(3)}:x^\a$. By Lemma \ref{differential_criterion}, we see that $x_1^2f$ and $x_2x_3f \in I$. Taking radical, we see that the right hand side contains the left hand side. Conversely, let $f$ be a minimal monomial generator of $(I:x_1) \cap (I:(x_2 x_3))$. Since $I\subseteq \sqrt{I^{(3)}:x^\a}$, we may assume that $\supp f$ is an independent set. By Lemma \ref{criterion_in_power}, it is easy to check that $f \in \sqrt{I^3:x^\a} \subseteq \sqrt{I^{(3)}:x^\a}$. The conclusion follows.
\end{proof}
If $x_2x_3 \in I$, then $\Delta_\a(I^{(3)}) = \Delta(I:x_1)$ is a cone over $1$. Since $F \cap \{1,2,3\} = \emptyset$, $\lk_{\Delta_\a (I^{(3)})} F$ is a cone over $1$, which is a contradiction. Thus, $x_2x_3 \notin I$. Let $\Gamma_1$ and $\Gamma_2$ be the simplicial complexes corresponding to $I:x_1$ and $I:(x_2x_3)$. Then $\Delta_\a(I^{(3)}) = \Gamma_1 \cup \Gamma_2$. Since $F \cap \{1,2,3\} = \emptyset$, $\lk_{\Gamma_1}F$ and $\lk_{\Gamma_2}F$ are cones over $1$ and $23$ respectively. By Lemma \ref{lem_MayerVietoris}, we must have $\h_{i-2}(\lk_{\Gamma_1 \cap \Gamma_2}F;\k) \neq 0$. Note that $\Gamma_1 \cap \Gamma_2 = I + N(\{1,2,3\})$. By Remark \ref{rem_extremal_set}, we must have $1,2,3 \in N(\{1,2,3\})$. This implies that $x_1x_2, x_1x_3\in I$. By Lemma \ref{lem_coveringedge}, $I + N(\{1,2,3\})$ is a simplex. Thus $i = 1$. Hence, $|\a| + i =5$ as required.

\noindent \textbf{Case 3.} $x^\a = x_1x_2x_3x_4$. Let $H$ be the restriction of $G$ to $\{1,2,3,4\}$. Since $G$ is gap-free, $H$ is gap-free. Since $x^\a \notin I^{(3)}$, by \cite[Theorem 2.10]{MNPTV}, $|E(H) | \le 5$. Denote $H^c$ the complement of $H$ in the complete graph on $1,2,3,4$. First, we claim that 
\begin{equation}\label{eq4_5}
 \sqrt{I^{(3)}:x^\a} = \cap_{e\in H^c} I:e.
\end{equation}
\begin{proof}[Proof of Eq. \eqref{eq4_5}] Let $f$ be a minimal generator of $\sqrt{I^{(3)}:x^\a}$. By Lemma \ref{differential_criterion}, there exists $t > 0$ such that $x_ix_j f^t \in I$ for all $i\neq j\in \supp \a$. Hence, $f \in I:(x_ix_j)$. Thus, the left hand side is contained in the right hand side. Conversely, let $f$ be a minimal generator of the right hand side. We may assume that $\supp f$ is an independent set. Since $H^c$ has at least one edge, $|N(f) \cap \supp \a| \ge 1$. If $|N(f) \cap \supp \a| \ge 3$ then $f \in \sqrt{I^3:x^\a}$ by Lemma \ref{criterion_in_power}. If $|N(f) \cap \supp \a| = 2$, say $N(f) \cap \supp \a = \{1,2\}$. Then we must have $x_3x_4 \in I$. By Lemma \ref{criterion_in_power}, $f \in \sqrt{I^3:x^\a}$. If $|N(f) \cap \supp \a | = 1$, say $N(f) \cap \supp \a = \{1\}$, then we must have $x_2x_3,x_3x_4,x_2x_4 \in I$. In other words, $x_2x_3x_4 \in I^{(2)}$. Hence $x^\a f \in I^{(2)} I \subseteq I^{(3)}$. The conclusion follows. 
\end{proof}

For each $e \in H^c$, let $\Gamma_e = \Delta(I:e)$. Then $\Gamma_e$ is a cone over $\supp e$. We have $\Gamma = \Delta_\a(I^{(3)}) = \cup_{e \in H^c} \Gamma_e$. Since $x^\a \notin I^{(3)}$, $|E(H^c)| \ge 1$. Furthermore, since $F \cap \{1,2,3,4\} = \emptyset$, $\lk_{\Gamma_e}F$ is a cone over $\supp e$. Since $\h_{i-1} (\lk_\Gamma F;\k) \neq 0$, we must have $|E(H^c)| \ge 2$. If $|E(H^c)| = 2$, then $\Gamma = \Gamma_1 \cup \Gamma_2$, where $\Gamma_1$ and $\Gamma_2$ corresponds to $I:e_1$, $I:e_2$, with $e_1, e_2$ are the two edges of $H^c$. By Lemma \ref{lem_MayerVietoris}, $\h_{i-2} (\lk_{\Gamma_1 \cap \Gamma_2}F;\k) \neq 0$. Since $\Gamma_1 \cap \Gamma_2 = \Delta(I:e_1 + I: e_2)$, for it not be a cone over any $t \in e_1 \cup e_2$ we must have $e_1 \subseteq N(e_2)$ and $e_2 \subseteq N(e_1)$. In this case, by Lemma \ref{lem_coveringedge}, $\Gamma_1 \cap \Gamma_2$ is a simplex. Thus, $i = 1$ and $|\a| + i = 5$. It remains to consider the cases where $|E(H^c) | \ge 3$, equivalently $|E(H)| \le 3$.

Subcase 3.a. $|E(H)| = 0$. In this case, we have 
\begin{equation}
    \Gamma = \Gamma_{12} \cup \Gamma_{13} \cup \Gamma_{14} \cup \Gamma_{23} \cup \Gamma_{24} \cup \Gamma_{34}.
\end{equation}
We see that $(\Gamma_{12} \cup \Gamma_{13} \cup \Gamma_{14} \cup \Gamma_{23} \cup \Gamma_{24}) \cap \Gamma_{34}, (\Gamma_{12} \cup \Gamma_{13} \cup \Gamma_{14} \cup \Gamma_{23}) \cap \Gamma_{24}, (\Gamma_{12} \cup \Gamma_{13} \cup \Gamma_{23}) \cap \Gamma_{14} $ are cones over $4$ and $(\Gamma_{12} \cup \Gamma_{13}) \cap \Gamma_{23}, \Gamma_{12} \cap \Gamma_{13}$ are cones over $3$. Since $F \cap \{1,2,3,4\} = \emptyset$, we have a contradiction by Lemma \ref{lem_MayerVietoris}.

Subcase 3.b. $|E(H)| = 1$. We may assume that $x_1x_2 \in G$. In this case, we have
\begin{equation}
    \Gamma = \Gamma_{13} \cup \Gamma_{14} \cup \Gamma_{23} \cup \Gamma_{24} \cup \Gamma_{34}.
\end{equation}
We see that $(\Gamma_{13} \cup \Gamma_{14} \cup \Gamma_{23} \cup \Gamma_{24}) \cap \Gamma_{34}, (\Gamma_{13} \cup \Gamma_{14} \cup \Gamma_{23}) \cap \Gamma_{24}, (\Gamma_{13} \cup \Gamma_{23}) \cap \Gamma_{14} $ are cones over $4$ and $\Gamma_{13} \cap \Gamma_{23}$ is a cone over $3$. Since $F \cap \{1,2,3,4\} = \emptyset$, we have a contradiction by Lemma \ref{lem_MayerVietoris}.

Subcase 3.c. $|E(H)| = 2$. Since $H$ is gap-free, we may assume that $x_1x_2,x_1x_3 \in H$. In this case, we have 
\begin{equation}
    \Gamma = \Gamma_{14} \cup \Gamma_{23} \cup \Gamma_{24} \cup \Gamma_{34}.
\end{equation}
We see that $(\Gamma_{14} \cup \Gamma_{23} \cup \Gamma_{24}) \cap \Gamma_{34}, (\Gamma_{14} \cup \Gamma_{23}) \cap \Gamma_{24}, \Gamma_{14} \cap \Gamma_{23}$ are cones over $4$. Since $F \cap \{1,2,3,4\} = \emptyset$, we have a contradiction by Lemma \ref{lem_MayerVietoris}.

Subcase 3.d. $|E(H)| = 3$. Since $H$ is gap-free, we may assume that $x_1x_2, x_1x_3 \in H$. There are three subcases.

3.d.$\alpha$. $x_1x_2,x_1x_3,x_1x_4 \in G$. We have 
\begin{equation}
    \Gamma = \Gamma_{23} \cup \Gamma_{24} \cup \Gamma_{34}.
\end{equation}
Since $(\Gamma_{23} \cup \Gamma_{24}) \cap \Gamma_{34}$ and $\Gamma_{23} \cap \Gamma_{24}$ are cones over $4$ and $F \cap \{1,2,3,4\} = \emptyset$, we have a contradiction by Lemma \ref{lem_MayerVietoris}.

3.d.$\beta$. $x_1x_2,x_1x_3,x_2x_3 \in G$. We have 
\begin{equation}
    \Gamma = \Gamma_{14} \cup \Gamma_{24} \cup \Gamma_{34}.
\end{equation}
Hence, $\Gamma$ is a cone over $4$, a contradiction.

3.d.$\gamma$. $x_1x_2,x_1x_3,x_3x_4 \in G$. We have 
\begin{equation}
    \Gamma = \Gamma_{23} \cup \Gamma_{14} \cup \Gamma_{24}.
\end{equation}
Since $\Gamma_{23}$ is a cone over $\{2,3\}$ and $\Gamma_{14} \cup \Gamma_{24}$ is a cone over $4$, by Lemma \ref{lem_MayerVietoris} we deduce that $\h_{i-2} (\lk_{(\Gamma_{14} \cup \Gamma_{24}) \cap \Gamma_{23}} F;\k) \neq 0$. Note that 
\begin{equation}
    \Gamma_{14} \cap \Gamma_{23} = \Delta(I+N[\{1,2,3,4\}]) \subseteq  \Delta(I+N(\{2,3,4\}))  = \Gamma_{24} \cap \Gamma_{23}.
\end{equation}
Hence, $(\Gamma_{14} \cup \Gamma_{24}) \cap \Gamma_{23} = \Gamma_{24} \cap \Gamma_{23}$ is a cone over $2$, a contradiction.
\end{proof}

\begin{thm} Let $I$ be the edge ideal of a simple graph $G$. Then $I^3$ has a linear resolution if and only if $G$ is gap-free and $\reg I \le 4$.
\end{thm}
\begin{proof} Assume that $G$ is gap-free and $\reg I \le 4$. By Theorem \ref{thmpow3}, $\reg I^3 = 6$. Thus, $I^3$ has a linear free resolution.

Conversely, assume that $I^3$ has a linear free resolution. By \cite[Proposition 1.3]{NP}, $G$ is gap-free. Now by Theorem \ref{thmpow3}, we have $\reg I^3 = \max \{ \reg I + 2, 6\} = 6$. Thus $\reg I \le 4$, as required.
\end{proof}

\begin{rem} It is interesting to find  a combinatorial characterization of gap-free graphs $G$ with $\reg I(G) \le 4$.
\end{rem}

\subsection*{Acknowledgment} We are grateful to Prof. Seyed Amin Seyed Fakhari for many discussions on the results of this paper. 


\begin{thebibliography}{2}


\bibitem[ABS]{ABS}
A. Alilooee, S. Beyarslan and S. Selvaraja,
{\em Regularity of powers of edge ideals of unicyclic graphs},
Rocky Mountain J. Math. {\bf 49}, no. 3, (2019), 699--728.


\bibitem [AB]{AB} 
A. Alilooee and A. Banerjee, 
{\em Powers of edge ideals of regularity three bipartite graphs},
J. Commut. Algebra {\bf 9} (2017), no. 4, 441--454. 



\bibitem[B]{B}	
A. Banerjee, 
{\it The regularity of powers of edge ideals},
J. Algbr. Comb. {\bf 41} (2015), no. 2, 303--321.
	
	
\bibitem[BH]{BH}
W. Bruns and J. Herzog,
\emph{Cohen-Macaulay rings. Rev. ed.}.
Cambridge Studies in Advanced Mathematics {\bf 39}, Cambridge University Press (1998).


\bibitem [BHT]{BHT} 
S. Beyarslan, H. T. Ha, and T. N. Trung,
{\it Regularity of powers of forests and cycles},
J. Algbr. Comb. {\bf 42} (2015), 1077--1095.


\bibitem[BN] {BN} 
A. Banerjee and E. Nevo, {\em Regularity of edge ideals via suspension}, 
arXiv:1908.03115, (2019).  


\bibitem[CHT]{CHT}
S.D. Cutkosky, J. Herzog, and N.V. Trung,
{\em Asymptotic behaviour of the Castelnuovo-Mumford regularity},
 Compositio Math. {\bf 118} (1999), no. 3, 243--261. 
 
 
\bibitem[CKV]{CKV}
A. Constantinescu, T. Kahle, M. Varbaro,
{\em Linear syzygies, flag complexes, and regularity},
 Collect. Math. {\bf 67} (2016), 357--362.
 
\bibitem[DHS]{DHS}
H. Dao, C. Huneke, and J. Schweig,
{\em Bounds on the regularity and projective dimension of
ideals associated to graphs},
J. Algbr. Comb. {\bf 38} (2013), no. 1, 37--55.


\bibitem[D]{D}
 R. Diestel, \emph{Graph theory, 2nd. edition,} Springer: Berlin/Heidelberg/New York/Tokyo, 2000.
 
 
 \bibitem[DHNT]{DHNT}
 L. X. Dung, T. T. Hien, H. D. Nguyen, and T. N. Trung,
 {\em Regularity and Koszul property of symbolic powers of monomial ideals}, Math. Z. {\bf 298} (2021), 1487--1522.
 
\bibitem[E]{E}
D. Eisenbud,
\emph{Commutative algebra. With a view toward algebraic geometry},
Graduate Texts in Mathematics {\bf 150}. Springer-Verlag, New York (1995).
 


\bibitem[Er]{Er}
	N. Erey,
	\emph{Powers of ideals associated to (C4,2K2)-free graphs},
	J. Pure Appl. Algebra {\bf 223} (2019), no. 7, 4007--4020.


\bibitem[EHHM]{EHHM}
N. Erey, J. Herzog, T. Hibi, and S. Madani,
{\em Matchings and squarefree powers of edge ideals},
 J. Comb. Theory, Series A {\bf 188} (2022), 105585.

\bibitem[F1]{F1}
	S. A. Seyed Fakhari, 
	{\em Regularity of symbolic powers of edge ideals of unicyclic graphs},
	J. Algebra {\bf 541} (2020), 345-358.
	
	
	\bibitem[F2]{F2}
	S. A. Seyed Fakhari, 
	{\em Regularity of symbolic powers of edge ideals of Cameron-Walker graphs,}
	Comm. Algebra {\bf 48}, Issue 12, (2020), 5215--5223.
	
	\bibitem[F3]{F3}
	S. A. Seyed Fakhari, 
	{\em On the regularity of small symbolic powers of edge ideals of graphs,}
	arXiv:1908.10845, August 2019.
	
	
 \bibitem [F]{F}
 R. Fr\"{o}berg,
 {\em On Stanley-Reisner rings}, 
 Topics in algebra, Banach Center Publications, {\bf 26} (1990), 57--70. 
 


	\bibitem [HTr]{HTr}
	L. T. Hoa, T. N. Trung, 
	{\it Castelnuovo-Mumford regularity of symbolic powers of two-dimensional square-free monomial ideals,} 
	J. Commut. Algebra, {\bf 8} Number 1 (2016), 77-88.
	
	\bibitem [HiTr]{HiTr}
	T. T. Hien, T. N. Trung, 
	{\it Regularity of symbolic powers of square-free monomial ideals}, arXiv:2108.06750


\bibitem [HHZ]{HHZ} 
J. Herzog, T. Hibi, X. Zheng, 
{\em Dirac's theorem on chordal graphs and Alexander duality},
Eur. J. Comb. {\bf 25} (2004), 949--960. 


	\bibitem [JK]{JK}
	A. V. Jayanthan, R. Kumar,
	{\em Regularity of Symbolic Powers of Edge Ideals},
	J. Pure Appl. Algebra {\bf 224} (2020), 106306. 


\bibitem[JNS]{JNS}
A. V. Jayanthan, N. Narayanan, and S. Selvaraja,
{\em Regularity of powers of bipartite graphs}, J. Algebr. Comb. {\bf 47} (2018), no. 1, 281--289.

	

\bibitem[JS] {JS} 
T. Januszkiewicz, J. Swiatkowski, {\em Hyperbolic coxeter groups of large dimension}, Comm. Math. Helv. {\bf 78} (2003), no. 3, 555--583.  


 
\bibitem[Kod] {Kod} 
V. Kodiyalam, 
{\em Asymptotic behaviour of Castelnuovo-Mumford regularity},
 Proc. Amer. Math. Soc. {\bf 128} (2000), 407--411.


\bibitem[LN] {LN} 
F. H. Lutz and E. Nevo, 
{\em Stellar theory for flag complexes},
 Math. Scand. {\bf 118}, (2016), no. 1, 70--82.

\bibitem[MNPTV]{MNPTV}
N. C. Minh, L. D. Nam, T. D. Phong, P. T. Thuy, and T. Vu,
{\em Comparison between regularity of small powers of symbolic powers and ordinary powers of an edge ideal,} J. Combin. Theory Ser. A {\bf 190} (2022), 105621.

	
\bibitem[MV1] {MV1} 
N. C. Minh and  T. Vu, 
{\em Survey on regularity of symbolic powers of an edge ideal}, In: Peeva, I. (eds) Commutative Algebra. Springer, Cham. \url{https://doi.org/10.1007/978-3-030-89694-2_18} (2021).

\bibitem[MV2]{MV2}
N. C. Minh and T. Vu, 
{\em Regularity of powers of Stanley-Reisner ideals of one-dimensional simplicial complexes}, to appear in Math. Nachr., \url{https://doi.org/10.1002/mana.202100485}, arXiv:2109.06396.



\bibitem[MV3]{MV3}
N. C. Minh and T. Vu, 
{\em Integral closure of powers of edge ideals and their regularity}, J. Algebra {\bf 609} (2022), 120--144.


\bibitem[NP] {NP} 
E. Nevo and I. Peeva,
{\em $C_4$-free edge ideals},
J. Algbr. Comb. {\bf 37} (2013), 243--248.


\bibitem[NV] {NV} 
H. D. Nguyen and T. Vu,
{\em Homological invariants of powers of fiber products,} 
	Acta Math. Vietnam. {\bf 44} (2019), 617--638.

\bibitem[S]{S} 
R. Stanley, 
{\em Combinatorics and Commutative Algebra}, 
2. Edition, Birkh$\ddot{\text{a}}$user, 1996.

\end{thebibliography}
\end{document}